\documentclass[12pt]{article}
\usepackage{amsthm}
\usepackage{amssymb,latexsym,amsmath}
\usepackage{amscd}

\oddsidemargin=\evensidemargin \footskip=36pt \voffset=-1.5cm
\def\tm#1{\item[{\rm (#1)}]}
\def\css{\begin{cases}}
\def\ecss{\end{cases}}

\def\sin{{\rm{sin}}}
\def\cos{{\rm{cos}}}

\def\bull{\vrule height 1.2ex width 1ex depth -.1ex }

\makeatletter
\renewcommand{\subsection}{\@startsection{subsection}{2}{0mm}{7mm}{4mm}
{\bf\normalsize}}

\def\nmrt{\begin{enumerate}}
\def\enmrt{\end{enumerate}}
\def\tm#1{\item[{\rm (#1)}]}

\makeatother
\newtheorem{formula}{}[section]

\newtheorem{remark}[formula]{Remark}
\newtheorem{lemma}[formula]{Lemma}
\newtheorem{theorem}[formula]{Theorem}

\newtheorem{example}[formula]{Example}

\begin{document}
\baselineskip=18pt
\date{}
\begin{center}
{\bf\Large A conjecture of Cameron and Kiyota on sharp characters with prescribed values}
\\ \vspace{1cm}
{\small A.~Abdollahi$^{a,c}$, J.~Bagherian$^a${\footnote{Corresponding author}}, M.~Khatami$^a$, Z. Shahbazi$^a$, R.~Sobhani$^{b}$} \\
 \vspace{-0.1cm}
{\small $^a$ Department of Pure Mathematics, Faculty of Mathematics and Statistics, University of Isfahan, Isfahan 81746-73441, Iran}\\
\vspace{-0.1cm}
{\small $^b$ Department of Applied Mathematics and Computer Science, Faculty of Mathematics and Statistics, University of Isfahan, Isfahan 81746-73441, Iran}\\
\vspace{-0.1cm}
{\small $^c$ School of Mathematics, Institute for Research in Fundamental Sciences (IPM), 19395-5746 Tehran, Iran}\\
{\small {\bf Emails:} a.abdollahi@sci.ui.ac.ir, bagherian@sci.ui.ac.ir, m.khatami@sci.ui.ac.ir, z.shahbazi@sci.ui.ac.ir, r.sobhani@sci.ui.ac.ir}
\vspace{1.5cm}
\end{center}



\begin{abstract}
Let $ \chi $ be a virtual (generalized) character of a finite group $ G $ and $ L=L(\chi)$ be the image of $ \chi $ on $ G-\lbrace 1 \rbrace $.
The pair  $ (G, \chi) $ is said to be sharp of type $ L $ if $|G|=\prod _{ l \in L} (\chi(1) - l) $.
If the principal character of $G$ is not an irreducible constituent of $\chi$, the pair $(G,\chi)$ is called normalized.
In this paper, we first provide some counterexamples to a conjecture that was proposed by Cameron and Kiyota in $1988$. This  conjecture states that
if $(G,\chi)$ is sharp and $|L|\geq 2$, then the inner product
$(\chi,\chi)_G$ is uniquely determined by $ L $. We then prove that this conjecture is true in the case that $(G,\chi) $ is normalized, $\chi$ is a character of $ G $, and $ L $
contains at least an irrational value.\\

\end{abstract}

\section{Introduction}
Let $ G $ be a finite group and $ \chi $ be a virtual (generalized) character of $ G $. We define
$$L(\chi):=\lbrace \chi(g)~~|~~1\neq g\in G \rbrace $$
and
$$ Sh(\chi) = \prod \limits_{ l \in L(\chi)} (\chi(1) - l). $$
It is known that $ Sh(\chi) $ is a multiple of $|G|$; see \cite{blich}.
The pair $ (G, \chi) $ (or briefly, the virtual character $ \chi $) is called sharp of type $ L $ or $ L $-sharp
if $ L=L(\chi) $ and $ Sh(\chi)=|G| $. Clearly,  $\chi$ is faithful whenever $(G,\chi)$ is $L$-sharp.
The pair $ (G, \chi) $ is said to be normalized if $ (\chi,1_{G})_{G}=0 $, where $1_G$ is the principal character of $G$.

The notion of sharpness was first introduced for permutation characters by Ito
and Kiyota  as a generalization of sharply $t$-transitive permutation representations \cite{ito}.
Then, Cameron and Kiyota extended the definition of sharp pairs to what is
given above and  posed the following  conjecture \cite{cam}.\\

\textbf{Conjecture 1}: If $ (G, \chi) $ is $ L $-sharp with $ |L|\geqslant 2 $, then $ (\chi, \chi)_{G} $ is uniquely determined by $ L $.\\

As a main result in this paper, we provide some counterexamples to the above conjecture.
In Section $2$, we give a group $G$ and two characters $\chi$ and
$\theta$ of $G$ such that the pairs $(G,\chi)$ and $(G,\theta)$ are $L$-sharp, but $(\chi,\chi)_G\neq (\theta, \theta)_G$ (see Example \ref{exm01}).
Moreover, we show that if $(G,\chi)$ in Conjecture $1$ is assumed to be normalized, then there are some counterexamples which show
 $(\chi,\chi)_G$ is not uniquely determined by $L$ (see Example \ref{exm02}).
It should be emphasized that we can not find any counterexample to Conjecture $1$ in the case that $(G,\chi)$ is
normalized and $\chi$ is a character. It seems that this conjecture is true whenever $ (G, \chi) $ is normalized and $\chi$ is a character.
Therefore, we give an improvement of Conjecture 1 as follows.\\

\textbf{Conjecture 2}: If $ (G, \chi) $ is $ L $-sharp and normalized where $\chi$ is a character and $ |L|\geqslant 2 $, then $ (\chi, \chi)_{G} $ is uniquely determined by $ L $.\\

In Section $3$, we show that Conjecture $2$ is true whenever $L$ contains at least an irrational value.
To do this, we use  a characterization of
$L$-sharp pairs, given in \cite{M1}, when $L$ contains at least an irrational value.

\section{$L$-Sharp characters with different numbers of irreducible constituents}
In this section, we first give a group $G$ with two characters $\chi$ and $\theta$ such that $(G,\chi)$ and $(G,\theta)$ are $L$-sharp and
$(\chi,\chi)_G\neq (\theta, \theta)_G$ (see Example \ref{exm01}). This shows that Conjecture $1$ is not true, in general.

\begin{example}\label{exm01}
{\rm Let $G$ be the group {\texttt{SmallGroup}(32,6)} in {\rm GAP }\cite{gap}.
Put $ \chi=\chi_{1}+2\chi_{2}+\chi_{5} $. Then as we can see in Table $1$, $ \chi $ is sharp of type $ L = \lbrace -1,3 \rbrace $ and
$ (\chi,\chi)_{G}=6 $. On the other hand, set $\theta=\chi_{2}+\chi_{3}+\chi_{4}+\chi_{5}$. Then $\theta$ is also sharp of type $L=\lbrace -1,3 \rbrace $ such that $ (\theta,\theta)_{G}=4 $. But we have $ (\chi,\chi)_{G} \neq (\theta,\theta)_{G} $.}
\begin{table}[!ht]
\centering
\begin{tabular}{c|crrccrrrrcr}
$ |Cl_{G}(g)| $ & $1$ & $4$ & $4~$ & $2~$ & $2$& $1$ & $4$ & $4$& $4~$ & $2$ & $4$ \\
\hline
$ \chi_{1} $ & $1$ & $1$ & $1~$ & $1~$ & $1$ & $ 1$ & $1$ & $1$ & $1~$ & $1$ & $1$\\
$ \chi_{2} $ & $1$ & $-1$ & $1~$ & $1~$ & $1$ & $1$ & $-1$ & $-1$ & $1~$ & $1$ & $-1$ \\
$ \chi_{3} $ & $1$ & $1$ & $-1~$ & $1~$ & $1$ & $ 1$ & $-1$ & $1$ & $-1~$ & $1$ & $-1$\\
$ \chi_{4} $ & $1$ & $-1$ & $-1~$ & $1~$ & $1$ & $1$ & $1$ & $-1$ & $-1~$ & $1$ & $1$\\
$ \chi_{5} $ & $4$ & $0$ & $0~$ & $0~$ & $0$ & $ -4$ & $0$ & $0$ & $0~$ & $0$ & $0$\\
\end{tabular}
\caption{Sharp characters $\chi=\chi_{1}+2\chi_{2}+\chi_{5} $ and $\theta=\chi_{2}+\chi_{3}+\chi_{4}+\chi_{5}$}
\end{table}
\end{example}

The example below shows that Conjecture $1$ is not true if we further assume that $(G,\chi)$ is normalized.
\begin{example}\label{exm02}
{\rm Let $G$ be the group {\texttt{SmallGroup}(192,1494)} in {\rm GAP }\cite{gap}.
Put $ \chi = \chi_{2}+\chi_{5} $. Then $\chi$ is sharp of type $ L = \lbrace -2, 0, 2 \rbrace $ and
$  (\chi, \chi)_{G} =2 $; see Table $2$. On the other hand, set $ \theta = \chi_{3}+\chi_{4}-\chi_{1} $. Then $\theta$ is also sharp of type $ L = \lbrace -2, 0, 2 \rbrace $ such that $ (\theta, \theta)_{G} = 3 $. Clearly,  $  (\chi, \chi)_{G} \neq (\theta, \theta)_{G}$.}
\begin{table}[!ht]
\centering
\begin{tabular}{c|rrrrrrrrrrrrr}
$ |Cl_{G}(g)| $ & $ 1 $ & $ 24 $ & $ 32 $ & $ 6 $ & $ 6 $ & $ 1 $ & $ 12 $ & $ 24 $ & $ 32 $ & $ 12~~ $ & $ 6 $ & $ 24 $ & $ 12 $ \\
\hline
$ \chi_{1} $ & $ 1 $ & $ -1 $ & $ 1 $ & $ 1 $ & $ 1 $ & $ 1 $ & $ -1 $ & $ -1 $ & $ 1 $ & $ 1 ~~$ & $ 1 $ & $ -1 $ & $ -1 $ \\
$ \chi_{2} $ & $ 2 $ & $ 0 $ & $ -1 $ & $ 2 $ & $ 2 $ & $ 2 $ & $ 0 $ & $ 0 $ & $ -1 $ & $ 2~~ $ & $ 2 $ & $ 0 $ & $ 0 $ \\
$ \chi_{3} $ & $ 3 $ & $ 1 $ & $ 0 $ & $ -1 $ & $ -1 $ & $ 3 $ & $ -1 $ & $ -1 $ & $ 0 $ & $ -1~~ $ & $ 3 $ & $1 $ & $ -1 $ \\
$ \chi_{4} $ & $ 4 $ & $ 0 $ & $ 1 $ & $ 0 $ & $ 0 $ & $ -4 $ & $ -2 $ & $ 0 $ & $ -1 $ & $ 0 ~~$ & $ 0 $ & $  0 $ & $ 2 $ \\
$ \chi_{5} $ & $ 4 $ & $ 0 $ & $ 1 $ & $ 0 $ & $ 0 $ & $ -4 $ & $ 2 $ & $ 0 $ & $ -1 $ & $ 0 ~~$ & $ 0 $ & $ 0 $ & $ -2 $
\end{tabular}
\caption{Sharp characters $ \chi = \chi_{2}+\chi_{5} $ and $ \theta = \chi_{3}+\chi_{4}-\chi_{1} $}
\end{table}
\end{example}

In the following, we give some remarks related to Conjecture $2$.
\begin{remark}
We can express a strong version of Conjecture $2$ as follows:\\

If $ (G, \chi) $ and $(H,\theta)$ are $ L $-sharp and normalized such that $\chi$ and $\theta$ are characters and $ |L|\geqslant 2 $, then
$ (\chi, \chi)_{G} =(\theta,\theta)_H$.\\

By Example \ref{rem12} below, we show that the above conjecture is not true, in general.
In fact, there exist non-isomorphic groups $ G $ and $ H $  with $L$-sharp
characters $ \chi $ and $ \theta $  with $ |L| \geq 2 $ such that $ (\chi, \chi)_{G} \neq(\theta, \theta)_{H} $.
Moreover, if we further assume that $|G|=|H|$, then there exists an counterexample to the above conjecture; see Example \ref{rem123}.
\end{remark}
\begin{example}\label{rem12}
{\rm Let $G$ be the alternating group $ A_{7} $ of degree $7$. Using the character table of $G$, one can check that $G$ has
an irreducible sharp character $\chi$ of type $ L = \lbrace-1, 0, 2 \rbrace $; see Table $3$.
On the other hand, let $H$ be the dihedral group $ D_{12} $ of order $12$.
Then $H$ has the sharp character $ \theta= \chi_{2}+\chi_{6} $  of type $ L = \lbrace-1, 0, 2 \rbrace $; see Table $4$.
Clearly, in this case $(\chi,\chi)_G\neq (\theta,\theta)_H$.}
\begin{table}[!ht]
\centering
\begin{tabular}{c|ccccccccc}
 $|Cl_{G}(g)|$ &  1  &  105 &  70  &  210  &  280  &  630  &  504  &  360  &  360  \\
\hline
 $\chi$  &  14  &  2  &  2  &  2  &  -1  &  0  &  -1  &  0  &  0
\end{tabular}
\caption{Sharp character $\chi$}
\end{table}
\begin{table}[!ht]
\centering
\begin{tabular}{c|rrrrrr}
$ |Cl_{G}(g)| $ & $ 1 $ & $ 1 $ & $ 3 $ & $ 3 $ & $ 2 $ & $ 2 $ \\
\hline
$ \chi_{1} $&$ 1 $ & $ 1 $ & $ 1 $ & $ 1 $ & $ 1 $ & $ 1 $ \\
$ \chi_{2} $&$ 1 $ & $ 1 $ & $ -1 $ & $ -1 $ & $ 1 $ & $ 1 $ \\
$ \chi_{3} $&$ 1 $ & $ -1 $ & $ -1 $ & $ 1 $ & $ 1 $ & $ -1 $ \\
$ \chi_{4} $&$ 1 $ & $ -1 $ & $ 1 $ & $ -1 $ & $ 1 $ & $ -1 $ \\
$ \chi_{5} $&$ 2 $ & $ 2 $ & $ 0 $ & $ 0 $ & $ -1 $ & $ -1 $ \\
$ \chi_{6} $&$ 2 $ & $ -2 $ & $ 0 $ & $ 0 $ & $ -1 $ & $ 1 $
\end{tabular}
\caption{Character table of $D_{12}$}
\end{table}
\end{example}

\begin{example}\label{rem123}
{\rm Let $G$ be the group {\texttt{SmallGroup}(192,955)} in {\rm GAP } \cite{gap}.
Then as we can see in Table $5$, $G$ has an irreducible sharp character $\chi$ of type $ L = \lbrace -2, 0, 2 \rbrace $.
\begin{table}[!ht]
\centering
\begin{tabular}{c|rrrrrrrrrrrrrr}
~ $ |Cl_{G}(g)| $ & $ 1 $ & $ 12 $ & $ 4 $ & $ 32 $ & $ 6 $ & $ 3 $ & $ 12 $ & $ 24 $ & $ 12 $ & $ 32 $ & $ 12 $ & $ 6 $ & $ 24 $ & $ 12 $ \\
\hline
$ \chi $&$ 6 $ & $ -2 $ & $ 0 $ & $ 0 $ & $ 2 $ & $ -2 $ & $ 0 $ & $ 0 $ & $ 2 $ & $ 0 $ & $ 0 $ & $-2 $ & $ 0 $ & $ 0 $
\end{tabular}
\caption{Sharp character $\chi$}
\end{table}
On the other hand, let $ H $ be the group {\texttt{SmallGroup}(192,1494)} in {\rm GAP } \cite{gap}.
Using the character table of $H$, $\theta = \theta_{1}+\theta_{2}$ is a sharp character of type $ L = \lbrace -2, 0, 2 \rbrace $ (see Table $6$).
Clearly, in this case $(\chi,\chi)_G\neq (\theta,\theta)_H$.}
\begin{table}[!ht]
\centering
\begin{tabular}{c|rrrrrrrrrrrrr}
~ $ |Cl_{H}(h)| $ & $~~~~1$ & $ 24 $ & $ 32 $ & $ 6 $ & $ 6 $ & $ 1 $ & $ 12 $ & $ 24 $ & $ 32 $ & $ 12 $ & $ 6 $ & $ 24 $ & $ 12 $ \\
\hline
$ ~~~~\theta_{1} $&$~~~~2$ & $ 0 $ & $ -1 $ & $ 2 $ & $ 2 $ & $ 2 $ & $ 0 $ & $ 0 $ & $ -1 $ & $ 2 $ & $ 2 $ & $0 $ & $ 0 $ \\
$ ~~~~\theta_{2} $&$~~~~4$ & $ 0 $ & $ 1 $ & $ 0 $ & $ 0 $ & $ -4 $ & $ 2 $ & $ 0 $ & $ -1 $ & $ 0 $ & $ 0 $ & $0 $ & $ -2 $
\end{tabular}
\caption{Sharp character $\theta = \theta_{1}+\theta_{2}$}
\end{table}
\end{example}

\section{Sharp characters with at least an irrational value }
In this section, we show that if $ (G, \chi) $ is $L$-sharp and normalized such that $\chi$ is a character and $L$
contains at least an irrational value, then Conjecture 2 holds. Our key tool is the following result given by  Alvis and  Nozawa in \cite{M1}.

For the remainder of this paper,
whenever $ (G, \chi) $ is sharp, we assume that $ \chi $ is a character of $ G $.\\

\begin{theorem}(\cite[Theorem 1.3]{M1})\label{b}
Suppose $ (G, \chi) $ is $L$-sharp and normalized such that $ L $ contains at least an irrational value. Then one of the following holds.
\nmrt
\tm{i} $ G $ is cyclic of order $ m $, and either $ m\geq 3 $ and $ \chi $ is linear, or else
$ m\geq 5 $ and $ \chi $ is the sum of two complex conjugate linear characters of $ G $.
\tm{ii} $ G $ is dihedral of order $ 2m $, where $ m\geq 5 $ is odd, and $ \chi $ is irreducible of degree $ 2 $.
\tm{iii}$ G $ is dihedral or generalized quaternion of order $ 2m $, where $ m\geq 8 $ is even, and $ \chi = \psi $ or $ \chi = \psi + \varepsilon $, where $ \psi $ is irreducible of degree $ 2 $ and $ \varepsilon $ is linear with cyclic kernel of order $ m $.
\tm{iv} $ G $ is isomorphic to the binary octahedral group and $ \chi $ is irreducible of degree $ 2 $.
\tm{v} $ G $ is isomorphic to $ SL(2, 5) $ and $ \chi $ is irreducible of degree $ 2 $.
\tm{vi}$ G $ is isomorphic to $ A_{5} $ and $ \chi $ is irreducible of degree $ 3 $.
\enmrt
\end{theorem}

The next lemmas give some results about the sharp characters of cyclic groups, dihedral groups, and generalized quaternion groups
which we need in the sequel.

\begin{lemma}(\cite[Proposition 1.8 ]{cam})\label{c}
Let $ G $ be a cyclic group of order $ m $ and $ \lambda $ be a faithful linear character of $ G $. Then $ (G, \lambda) $ is sharp of type
$$ L = \lbrace \omega^{r}~~|~~1\leq r \leq m-1 \rbrace $$ where $ \omega = e^{2\pi i/m} $.
\end{lemma}

\begin{lemma}\label{d}
Let $ G $ be a cyclic group of order $ m $,  where $m \geq 5 $ is odd,  and let $ \lambda $ be a faithful linear character of $ G $.
Then $ (G, \lambda+\overline{\lambda} ) $ is a sharp pair of type $ L = \lbrace \omega^{r}+\omega^{-r}~~|~~1\leq r \leq (m-1)/2 \rbrace $, where $ \omega = e^{2\pi i/m} $.
\end{lemma}
\begin{proof}
Set $ \chi = \lambda+\overline{\lambda} $.
First we show that $ \chi $ is faithful. Let $ L = \lbrace \chi(g)~~|~~1\neq g\in G \rbrace $. Then $L= \lbrace 2 ~\cos(2\pi r/m)~~|~~1\leq r \leq (m-1)/2 \rbrace $. If there exists $ 1\leq r \leq (m-1)/2 $ such that $ 2 ~\cos(2\pi r/m) = 2 $,
then $ 2\pi r/m = 2k\pi $ for some $ k\in \mathbb{Z} $, a contradiction. 
Hence $ \chi $ is faithful. Moreover,
\begin{align*}
\prod \limits_ { r=1 }^{ (m-1)/2} \Big( 2 - 2~\cos\Big(\frac{2\pi r}{m}\Big)\Big) &=\prod \limits_ { r=1 }^{ (m-1)/2} 4~\sin^ 2 \Big({\frac{\pi r}{m}}\Big)  \\
&= \prod \limits_{ r=1 }^{ (m-1)/2} 2~\sin\Big(\frac{\pi r}{m}\Big) \prod \limits_{ r=1 }^{ (m-1)/2} 2~\sin\Big(\frac{\pi r}{m}\Big)  \\
&=  \prod \limits_{ r=1 }^{ (m-1)/2} 2~\sin\Big(\frac{\pi r}{m}\Big) \prod \limits_{ r=1 }^{ (m-1)/2} 2~\sin\Big(\frac{(m-r)\pi }{m}\Big)\\
&=\prod \limits_{ r=1 }^{ (m-1)/2} 2~\sin\Big(\frac{\pi r}{m}\Big) \prod \limits_{ j=(m+1)/2 }^{ m-1} 2~\sin\Big(\frac{\pi j }{m}\Big)\\
&=\prod \limits_{ r=1 }^{m-1} 2~\sin\Big(\frac{\pi r}{m}\Big) = m,
\end{align*}
(See \cite[Lemma 3.1]{MT}). This shows that $ \chi $ is a sharp character, as desired. \hfill\bull
\end{proof}

\begin{lemma}\label{f}
Let $ D_{2m}$ be the dihedral group of order $ 2m $, where $ m $ is odd, and let $ \chi $ be a faithful
irreducible character of $D_{2m} $ of degree $ 2 $. Then $ (D_{2m}, \chi) $ is a sharp pair of
type $ L = \lbrace 0,~\omega^{r}+ \omega^{-r}~~|~~1\leq r \leq (m-1)/2 \rbrace $, where $ \omega = e^{2\pi i/m} $.
\end{lemma}
\begin{proof}
See the proof of Theorem 3.2 in \cite{MT}. \hfill\bull
\end{proof}
\begin{lemma}\label{g}
Let $ G $ be a dihedral or generalized quaternion group of order $ 2m $ where  $  m $ is even and let $\omega = e^{2\pi i/m} $. Then
\nmrt
\tm{1}\label{g_{1}} if $ m/4 \in \mathbb{Z} $ and $ \psi $ is a faithful irreducible character of degree $ 2 $, then $ (G, \psi) $ is a sharp pair of type $ L = \lbrace -2,~ 0, ~ \omega^{r}+ \omega^{-r}~~|~~1\leq r \leq m/2-1 \rbrace $; \\
\tm{2}\label{g_{2}} if $ \psi $ is a faithful irreducible character of degree $ 2 $ and $ \varepsilon $ is linear with cyclic kernel of order $ m $,
then $(G, \psi+\varepsilon) $ is a sharp pair of type $$ L = \lbrace -1, ~ 1+\omega^{r}+ \omega^{-r}~~|~~1\leq r \leq m/2-1 \rbrace .$$
\enmrt
\end{lemma}
\begin{proof}
Assume that $m=2t$ for some integer $ t $.  If $ G =D_{2m}$, then
$$ G= \langle~ a, ~b~~:~~a^{m} = b^{2} = 1,~~b^{-1}ab = a^{-1}~\rangle $$ and
it is known that Table $ 7 $ is the character table of $ G $ where
\begin{center}
$ \lbrace 1 \rbrace,~ \lbrace a^{t} \rbrace,~ \lbrace a^{r}, a^{-r} \rbrace $~for $~1\leq r \leq t-1~ , \lbrace a^{s}b~\vert ~ s~ even \rbrace, \lbrace a^{s}b~\vert ~ s~ odd \rbrace $
\end{center}
are the conjugacy classes of $ G $.
\begin{table}[!ht]\label{T2}
\centering
\begin{tabular}{||c|cccrr||}
\hline
$ Cl_G(g) $ & $ 1 $ & $ a^{t} $ & $ a^{r}(1\leq r \leq t-1) $ & $ b $ & $ ab $  \\
$ |Cl_G(g)|$ & $ 1 $ & $ 1 $ & $ 2 $ & $ m/2 $ & $ m/2 $\\
\hline
$ \chi_{1} $ & $ 1 $ & $ 1 $ & $1$ & $ 1 $ & $ 1 $ \\
$ \chi_{2} $ & $ 1 $ & $ 1 $ & $1$ & $ -1 $ & $ -1 $ \\
$ \chi_{3} $ & $ 1 $ & $ (-1)^{t} $ & $ (-1)^{r} $ & $ 1 $ & $  -1$\\
$ \chi_{4} $ & $ 1 $ & $ (-1)^{t} $ & $ (-1)^{r} $ & $ -1 $ & $ 1 $\\
$ \psi_{j} $ & $ 2 $  & $ 2(-1)^{j} $ & $ \omega^{jr}+\omega^{-jr} $ & $ 0 $ & $ 0 $\\
$ (1\leq j \leq t-1) $&&&&&\\
\hline
\end{tabular}
\caption{Character table of $ D_{2m} $ $( m $ even $ ) $}
\end{table}
Similarly, if  $ G = Q_{2m} $ is the generalized quaternion group of order $ {2m} $, then
$$ G = \langle~ a,~b~~:~~a^{2t}=1,~~~a^{t} = b^{2} ,~~~b^{-1}ab = a^{-1}~\rangle .$$ \\
In this case
$$ \lbrace 1 \rbrace,~ \lbrace a^{t} \rbrace,~ \lbrace a^{r}, a^{-r}\rbrace $$ for $ ~1\leq r \leq t-1$, and  $$~\lbrace a^{2r}b~\vert ~ 0\leq r\leq t-1 \rbrace, ~\lbrace a^{2r+1}b~\vert ~0\leq r\leq t-1 \rbrace $$
are the conjugacy classes of $ G $ and the character tables of the groups
$Q_{4t}$ ($t$ odd) and  $Q_{4t}$ ($t$ even) are given
in Table $8$ and Table $9$, respectively; (see \cite{M}).
Now suppose that $G$ is one of the two groups $D_{2m}$ or $ Q_{2m} $,  $ ( m $ even$ ) $.
We show that the irreducible character $ \psi_{1} $ is a sharp character of $ G $ of
type $$ L = \lbrace -2, ~0, ~ \omega^{r}+ \omega^{-r}~~|~~1\leq r \leq m/2-1 \rbrace .$$
It is easy to see that $ \psi_{1} $ is faithful and we have
\begin{align*}
\prod \limits_ { r=1 }^{ m/2-1} \Big( 2 - 2~\cos\Big(\frac{2\pi r}{m}\Big)\Big) & = \prod \limits_ { r=1 }^{ m/2-1 } 4~\sin^ 2 \Big({\frac{\pi r}{m}}\Big) \\
&= \prod \limits_{ r=1 }^{ m/2-1} 2~\sin\Big(\frac{\pi r}{m}\Big) \prod \limits_{ r=1 }^{ m/2-1} 2~\sin\Big(\frac{\pi r}{m}\Big)\\
&= \prod \limits_{ r=1 }^{ m/2-1} 2~\sin\Big(\frac{\pi r}{m}\Big) \prod \limits_{ r=1 }^{ m/2-1} 2~\sin\Big(\frac{(m-r)\pi }{m}\Big)\\
&=\prod \limits_{ r=1 }^{m/2-1 } 2~\sin\Big(\frac{\pi r}{m}\Big) \prod \limits_{ j=m/2+1 }^{ m-1} 2~\sin\Big(\frac{\pi j }{m}\Big) \\
&=\dfrac{\prod \limits_{ r=1 }^{m-1} 2~\sin\Big(\frac{\pi r}{m}\Big)}{{2~\sin(\dfrac { \pi}{2}})} = \dfrac{m}{2},
\end{align*}
(See \cite[Lemma 3.1]{MT}). Since for $ r=m/4 $, we have $ 2~\cos(\frac{2\pi r}{m}) = 0 $ it follows that
$ \prod _{ l \in L(\psi_{1})} (2 - l) = 2m $ and hence $ \psi_{1} $ is a sharp character of $G$ of type
$$ L = \lbrace -2, ~0, ~ \omega^{r}+ \omega^{-r}~~|~~1\leq r \leq m/2-1 \rbrace .$$
\begin{table}[!ht] \label{T3}
\centering
\begin{tabular}{||c|crcrr||}
\hline
$ Cl_G(g) $ & $ 1 $ & $ a^{t} $ & $ a^{r}(1\leq r \leq t-1) $ & $ b $ & $ ab $  \\
$ |Cl_G(g)| $ & $ 1 $ & $ 1 $ & $ 2 $ & $ m/2 $ & $ m/2 $\\
\hline
$ \chi_{1} $ & $ 1 $ & $ 1 $ & $ 1 $ & $ 1 $ & $ 1 $ \\
$ \chi_{2} $ & $ 1 $ & $ 1 $ & $ 1 $ & $ -1 $ & $ -1 $ \\
$ \chi_{3} $ & $ 1 $ & $ -1 $ & $ (-1)^{r} $ & $ i $ & $  -i $\\
$ \chi_{4} $ & $ 1 $ & $ -1 $ & $ (-1)^{r} $ & $ -i $ & $ i $\\
$ \psi_{j} $ & $ 2 $  & $ 2(-1)^{j} $ & $ \omega^{jr}+\omega^{-jr} $ & $ 0 $ & $ 0 $\\
$ (1\leq j \leq t-1) $&&&&&\\
\hline
\end{tabular}
\caption{Character table of $ Q_{4t} $ $( t $ odd $ ) $}
\end{table}
\begin{table}[!ht]\label{T4}
\centering
\begin{tabular}{||c|cccrr||}
\hline
$ Cl_G(g) $ & $ 1 $ & $ a^{t} $ & $ a^{r}(1\leq r \leq t-1) $ & $ b $ & $ ab $  \\
$ |Cl_G(g)| $ & $ 1 $ & $ 1 $ & $ 2 $ & $ m/2 $ & $ m/2 $\\
\hline
$ \chi_{1} $ & $ 1 $ & $ 1 $ & $ 1 $ & $ 1 $ & $ 1 $ \\
$ \chi_{2} $ & $ 1 $ & $ 1 $ & $ 1 $ & $ -1 $ & $ -1 $ \\
$ \chi_{3} $ & $ 1 $ & $ 1 $ & $ (-1)^{r} $ & $ 1$ & $ -1 $\\
$ \chi_{4} $ & $ 1 $ & $ 1 $ & $ (-1)^{r} $ & $ -1 $ & $ 1 $\\
$ \psi_{j} $ & $ 2 $  & $ 2(-1)^{j} $ & $ \omega^{jr}+\omega^{-jr} $ & $ 0 $ & $ 0 $\\
$ (1\leq j \leq t-1) $&&&&&\\
\hline
\end{tabular}
\caption{Character table of $ Q_{4t} $ $( t $ even $ ) $}
\end{table}

Now we consider irreducible character $ \psi_{j} $ where $ 2\leq j \leq t-1 $. Then one can see that
$ \psi_{j} $ is faithful if and only if $ j $ is odd and $ ( j, m ) = 1 $. In this case we have $ L(\psi_{1}) = L(\psi_{j}) $ and
the argument above shows that $\psi_j$ is a sharp character for $G$
of type $ L = \lbrace -2, ~0, ~ \omega^{r}+ \omega^{-r}~~|~~1\leq r \leq m/2-1 \rbrace $.

Similarly,  we can show that $ \chi_{2}+ \psi_{j} $, where $ 1\leq j \leq t-1 $ is odd and $ (j, m) = 1 $, is a sharp character
of type $ L = \lbrace -1, ~ 1+\omega^{r}+ \omega^{-r}~~|~~1\leq r \leq m/2-1 \rbrace $ for $G$. \hfill\bull
\end{proof}

\vspace{1cm}
For the remainder of this paper we put $ \omega = e^{2\pi i/m}$ and we assume that
\nmrt
\tm{a} $ L_{1} = \lbrace \omega^{r}~~|~~1 \leq r \leq m-1 \rbrace $ such that $ m\geq 3 $;\\
\tm{b} $ L_{2} = \lbrace \omega^{r}+\omega^{-r}~~|~~1\leq r \leq (m-1)/2 \rbrace $ such that $ m\geq 5 $ is odd;\\
\tm{c} $ L_{3} = \lbrace 0, ~ \omega^{r}+ \omega^{-r}~~|~~1\leq r \leq (m-1)/2 \rbrace $ such that $ m\geq 5 $ is odd;\\
\tm{d}$ L_{4} = \lbrace -2,~ 0, ~ \omega^{r}+ \omega^{-r}~~|~~1\leq r \leq m/2-1 \rbrace $ such that $ m \geqslant 8 $ is even; \\
\tm{e}   $ L_{5} = \lbrace -1, ~ 1+\omega^{r}+ \omega^{-r}~~|~~1\leq r \leq m/2-1 \rbrace $ such that $ m \geqslant 8 $ is even;\\
\tm{f}$ L_{6} = \lbrace -2,-1,0,1,\sqrt{2},-\sqrt{2} \rbrace $;\\
\tm{g} $ L_{7} = \lbrace -2,-1,0,1,\dfrac{1+\sqrt{5}}{2},\dfrac{1-\sqrt{5}}{2},\dfrac{-1+\sqrt{5}}{2},\dfrac{-1-\sqrt{5}}{2} \rbrace $;
\\
\tm{h} $ L_{8} = \lbrace -1,0,\dfrac{1+\sqrt{5}}{2},\dfrac{1-\sqrt{5}}{2} \rbrace $.\\
\enmrt
Then we can see that $ L_{i} \cap \mathbb{Z} \subseteq \lbrace -2, -1, 0, 1, 2 \rbrace $ for $ 1 \leq i \leq 5 $.
In the next lemmas for cyclic groups, dihedral groups, and  generalized quaternion groups we determine
$L_{i} \cap \mathbb{Z} $ whenever $ (G, \chi) $ is $ L_i $-sharp.

\begin{lemma}\label{j}
Let $ G $ be a cyclic group of order $ m $ and $ (G, \chi) $ be $ L $-sharp with $ L = L_{1} $ or $ L = L_{2} $, then $ L \cap \mathbb{Z} \subseteq \lbrace -1 \rbrace $.
\end{lemma}
\begin{proof}
First let $ L = L_{1} $. Since $ L_{1} $ is the set of $ m $th roots of unity (distinct from 1), it is
clear that $ L_{1} \cap \mathbb{Z} \subseteq \lbrace -1 \rbrace $.\\
Now suppose that $ L = L_{2} $. Since $ m $ is odd, it is easy to see that $ -2, 0, 1, 2 \notin L_{2} $. Moreover,
$ \omega^{r}+\omega^{-r} = -1 $ if and only if $ m/3 \in \mathbb{Z} $. So
$ L_{2} \cap \mathbb{Z} \subseteq \lbrace -1 \rbrace $. \hfill\bull
\end{proof}

\begin{lemma}\label{k}
Let $ G $ be a dihedral group of order $ 2m $, $ ( m\geq 5 $ odd $ ) $ and $ (G, \chi) $ be $ L $-sharp
with $ L = L_{3} $, then $ \lbrace 0 \rbrace \subseteq L \cap \mathbb{Z} \subseteq \lbrace -1,0 \rbrace $.
\end{lemma}
\begin{proof}
Since $ 0\in L_{3} \cap \mathbb{Z} $, it suffices to show that $ -2, 1, 2 \notin L_{3} $. If $ -2, 1 \in L_{3} $,
then we obtain a contradiction under assumption $ m\geq 5 $ is odd. Moreover, $ r\leq (m-1)/2 $ shows that $ 2 \notin L_{3} $.
Note that $ \omega^{r}+\omega^{-r} = -1 $ if and only if $ m/3 $ is an integer. Therefore if $ m/3 $ is integer,
then $ L \cap \mathbb{Z} = \lbrace -1, 0 \rbrace $. \hfill\bull
\end{proof}

\begin{lemma}\label{l}
Let $ G $ be a dihedral or generalized quaternion group of order $ 2m $, $ (m\geq 8 $ even$)$ and $ (G, \chi) $ be $ L $-sharp with $ L = L_{4} $, then $ L \cap \mathbb{Z} =\lbrace -2, 0 \rbrace $ or $ L \cap \mathbb{Z} = \lbrace -2, -1, 0, 1 \rbrace $.
\end{lemma}
\begin{proof}
Obviously, $ -2, 0 \in L_{4} $. Since $ r\leq m/2-1 $, we have $ 2 \notin L_{4} $. It is easy to see that $ -1, 1 \in L_{4} $ if and only if $ m/3 $ is an integer. Hence $ L_{4} \cap \mathbb{Z} =\lbrace -2, 0 \rbrace $ or $ L_{4} \cap \mathbb{Z} = \lbrace -2, -1, 0, 1 \rbrace $. \hfill\bull
\end{proof}

\begin{lemma}\label{m}
Let $ G $ be a dihedral or generalized quaternion group of order $ 2m $, $ (m\geq 8 ~ even)$ and $ (G, \chi) $
be $ L $-sharp  with $ L = L_{5} $. Then $ L \cap \mathbb{Z} \in \lbrace \lbrace -1\rbrace, \lbrace -1, 1 \rbrace, \lbrace -1, 0, 2 \rbrace, \lbrace -1, 0, 1, 2 \rbrace \rbrace $.
\end{lemma}
\begin{proof}
It follows from the assumption that $ -1 \in L_{5} $ and $ -2 \notin L_{5} $. Since $ m $ is even, we have $ 0, 2 \in L_{5} $
if and only if $ m/3 $ is an integer. Moreover, $ 1\in L_{5} $ if and only if $ r=m/4 $.
Therefore $ L \cap \mathbb{Z} $ is equal to $ \lbrace -1 \rbrace $, $ \lbrace -1, 1 \rbrace $,
$ \lbrace -1, 0, 2 \rbrace $ or $ \lbrace -1, 0, 1 , 2 \rbrace $. \hfill\bull
\end{proof}

\begin{theorem}\label{h}
Suppose $ (G, \chi) $ is $ L $-sharp and normalized such that $ L$ contains at least an
irrational value. If $ L \subseteq \mathbb{R} $ and
$$ L \cap \mathbb{Z} \in \big \lbrace \emptyset, \lbrace -1 \rbrace, \lbrace -1, 1 \rbrace, \lbrace -1, 0, 2 \rbrace, \lbrace -1, 0, 1, 2 \rbrace \big \rbrace, $$ then $ (\chi, \chi)_{G} = 2 $. Otherwise $ (\chi, \chi)_{G} = 1 $.
\end{theorem}
\begin{proof}
First suppose that $ L \subseteq \mathbb{R} $ and
$$ L \cap \mathbb{Z} \in \big \lbrace \emptyset, \lbrace -1 \rbrace, \lbrace -1, 1 \rbrace, \lbrace -1, 0, 2 \rbrace , \lbrace -1, 0, 1, 2 \rbrace \big \rbrace.$$
Then it follows from Theorem \ref{b} and Lemmas \ref{j}, \ref{k}, \ref{l}, and \ref{m} that
\nmrt
\tm{1} either $ G $ is a cyclic group of odd order and $ \chi $ is the sum of two complex conjugate linear characters of $ G $;
\tm{2} or $ G $ is a dihedral or generalized quaternion group of order $ 2m $~$ (~m\geq 8~even) $ and $ \chi $ is the sum of an irreducible character of degree $ 2 $ and a linear character with cyclic kernel of order $ m $.
\enmrt
So in each of these two cases we have $ (\chi, \chi)_{G} = 2 $.

Now we consider the other cases as follows.
\nmrt
\tm{i} If $ L \nsubseteq \mathbb{R} $, then by Theorem \ref{b} and Lemmas \ref{c}, \ref{d}, \ref{f}, and \ref{g}, $ G $ is a cyclic group and $ \chi $ is linear.
\tm{ii} If $ L \subseteq \mathbb{R} $ and $ L \cap \mathbb{Z} \not\in \big \lbrace \emptyset, \lbrace -1 \rbrace, \lbrace -1, 1 \rbrace, \lbrace -1, 0, 2 \rbrace , \lbrace -1, 0, 1, 2 \rbrace \big \rbrace $, then by Theorem \ref{b} and Lemmas \ref{j}, \ref{k}, \ref{l}, and \ref{m},
we have one of the following cases:
\begin{itemize}
\item[$ (a) $] $ G $ is dihedral of order $ 2m $, where $ m\geq 5 $ is odd, and $ \chi $ is irreducible of degree $ 2 $.
\item[$ (b) $] $ G $ is dihedral or generalized quaternion of order $ 2m $, where $ m\geq 8 $ is even, and $ \chi $ is irreducible of degree $ 2 $.
\item[$ (c) $] $ G $ is isomorphic to the binary octahedral group and $ \chi $ is irreducible sharp character of degree $ 2 $ and type $ L=L_{6} $.
\item[$ (d) $] $ G $ is isomorphic to $ SL(2, 5) $ and $ \chi $ is irreducible sharp character of degree $ 2 $ and type $ L=L_{7} $.
\item[$ (e) $] $ G $ is isomorphic to $ A_{5} $ and $ \chi $ is irreducible sharp character of degree $ 3 $ and type $ L=L_{8} $.
\end{itemize}
\enmrt
As we see in the cases $(i)$ and $(ii)$ above, the sharp character $\chi$ is irreducible and so we have $ (\chi, \chi)_{G} = 1 $. \hfill\bull
\end{proof}
\vspace{1cm}
Now we are ready to state and prove the main result of this section.
\begin{theorem}\label{i1}
If $ (G,\chi) $ is $ L $-sharp and normalized with $ |L| \geq 2 $ and $ L $ contains at least an irrational value,
then $ (\chi, \chi)_{G} $ is uniquely determined by $ L $.
\end{theorem}
\begin{proof}
It follows from Theorem \ref{h} that if $ (G,\chi) $ is $ L $-sharp with $ |L| \geq 2 $ and $ L $
contains at least an irrational value, then $  (\chi, \chi)_{G} \in \lbrace 1, 2 \rbrace $ and it can be uniquely determined by $ L $. \hfill\bull
\end{proof}
\begin{remark}
The example below shows that  the pair $ (G,\chi) $ in Theorem \ref{i1} must be normalized as  a necessary condition.
\end{remark}
\begin{example}\label{akh}
{\rm Let $ G $ be the dihedral group $ D_{16} $. Then $ \chi=\chi_{3}+2\chi_1 $ is a sharp character of
type $ \lbrace 2, 0,\sqrt{2}+2,-\sqrt{2}+2 \rbrace $; see Table $10$. On the other hand, $ \theta = \chi_{2}+\chi_{3}+\chi_{1} $
is also a sharp character of type $ \lbrace 2, 0,\sqrt{2}+2,-\sqrt{2}+2 \rbrace $. But we have  $(\chi,\chi)_G \neq (\theta, \theta)_{G}$.}
\begin{table}[!ht]
\centering
\begin{tabular}{c|crcrrrr}
$ |Cl_{G}(g)| $ & $1$ & $4$ & $2$ & $2$ & $1$ & $4$ & $2$ \\
\hline
$ \chi_{1} $ & $1$ & $1$ & $1$ & $1$ & $1$ & $1$ & $1$\\
$ \chi_{2} $ & $1$ & $-1$ & $1$ & $1$ & $1$ & $-1$ & $1$\\
$ \chi_{3} $ & $2$ & $0$ & $\sqrt{2}$ & $0$ & $-2$ & $0$ & $-\sqrt{2}$
\end{tabular}
\caption{Sharp characters $\chi=\chi_3+2\chi_1$ and $ \theta = \chi_{2}+\chi_{3}+\chi_{1} $}
\end{table}
\end{example}

\end{document}